\documentclass[11pt, reqno]{amsart}

\usepackage{color}
\usepackage{amsmath}
\usepackage{amssymb}
\usepackage{amsfonts}
\usepackage[lmargin=3cm,rmargin=3cm,tmargin=3cm,bmargin=3cm]{geometry}

\def\R{\mathbb{R}}

\numberwithin{equation}{section}

\newcommand{\EP}{\exp L^p}
\newcommand{\EPO}{\exp L^p_0}

\newtheorem{theorem}{Theorem}[section]
\newtheorem{lemma}[theorem]{Lemma}
\newtheorem{proposition}[theorem]{Proposition}
\newtheorem{corollary}[theorem]{Corollary}

\newtheorem{definition}[theorem]{Definition}

\newtheorem{remark}[theorem]{Remark}
\newtheorem{remarks}[theorem]{Remarks}

\author[M. Majdoub]{Mohamed Majdoub}
\address{Mohamed Majdoub\\ Department of Mathematics\\ College of Science\\ Imam Abdulrahman Bin Faisal University\\
P.O. Box 1982, Dammam, Saudi Arabia}
\email{\sl mmajdoub@iau.edu.sa}
\author[S. Tayachi]{Slim Tayachi}
\address{Slim Tayachi\\ Universit\'e de Tunis El Manar\\ Facult\'e des Sciences de Tunis\\ D\'epartement de Math\'ematiques\\ Laboratoire \'equations aux d\'eriv\'ees partielles (LR03ES04), 2092 Tunis, Tunisie} \email{\sl slim.tayachi@fst.rnu.tn,  slimtayachi@gmail.com}


\title[The heat equation with power-exponential nonlinearities]{Well-posedness, Global existence and decay estimates for the heat equation with general power-exponential nonlinearities}

\begin{document}
\begin{abstract} In this paper we consider the problem: $\partial_{t} u- \Delta u=f(u),\; u(0)=u_0\in \exp L^p(\R^N),$ where $p>1$ and $f : \R\to\R$ having an exponential growth at infinity with $f(0)=0.$ We prove local well-posedness in $\exp L^p_0(\R^N)$ for $f(u)\sim \mbox{e}^{|u|^q},\;0<q\leq p,\; |u|\to \infty.$ However,  if for some $\lambda>0,$ $\displaystyle\liminf_{s\to \infty}\left(f(s)\,{\rm{e}}^{-\lambda s^p}\right)>0,$ then non-existence occurs in $\exp L^p(\R^N).$ Under smallness condition on the initial data and for exponential nonlinearity $f$ such that  $|f(u)|\sim |u|^{m}$ as $u\to 0,$ ${N(m-1)\over 2}\geq p$, we show that the solution is global. In particular, $p-1>0$ sufficiently small is allowed. Moreover, we obtain decay estimates in Lebesgue spaces for large time which depend on $m$.
\end{abstract}
\maketitle

\section{Introduction}


In this paper we study  the Cauchy problem:
\begin{equation}\label{1.1}
\left\{\begin{array}{cc}
\partial_{t} u- \Delta u=f(u),   \\
u(0)=u_0\in \exp L^p(\R^N),
\end{array}
\right.
\end{equation}
where $p>1$ and $f : \R\to\R$ having an exponential growth at infinity with $f(0)=0.$

As is a standard practice, we study \eqref{1.1} via the associated integral equation:
\begin{equation}
\label{integral}
u(t)= {\rm e}^{t\Delta}u_{0}+\int_{0}^{t}{\rm e}^{(t-s)\Delta}\,f(u(s))\,ds,
\end{equation}
where ${\rm e}^{t\Delta}$ is the linear heat semi-group. {The Cauchy problem \eqref{1.1} has been extensively studied in the scale of Lebesgue spaces, especially for polynomial type nonlinearities. It is known that in this case one can always find a Lebesgue space $L^q,\; q<\infty$ for which \eqref{1.1} is locally well-posed. See for instance \cite{Br, HW, W, Indiana}.}

By analogy with the Lebesgue spaces, which are well-adapted to  the heat equations  with power  nonlinearities (\cite{W2}), we are motivated to consider the Orlicz spaces, in order to study  heat equations with power-exponential nonlinearities. Such spaces were introduced by Birnbaum and Orlicz \cite{BiOr} as a natural generalization of the classical Lebesgue spaces $L^q,\, 1<q<\infty$. For this generalization the function $x^q$ entering in the definition of $L^q$ space is replaced by a more general convex function: in particular ${\rm e}^{x^q}-1$.

 For the particular case where $f(u)\sim {\rm e}^{|u|^2},\; u$ large, well-posedness  results are proved in the Orlicz space $\exp L^2(\R^N).$ See \cite{IJMS, Ioku, IRT, RT}. It is also proved that if $f(u)\sim {\rm e}^{|u|^s},\, s>2,\; u$ large then the  existence is no longer guaranteed and in fact there is nonexistence in the Orlicz space  $\exp L^2(\R^N).$  See \cite{IRT}. Global existence and decay estimates are also established for the nonlinear heat equation with $f(u)\sim {\rm e}^{|u|^2},\; u$ large. See \cite{Ioku, MOT, FKRT}.

 Here we consider the general case $f(u)\sim {\rm e}^{|u|^q},\;q>1,$  $u$ large. For such exponential nonlinearities, the most adaptable space is the so-called Orlicz space  $\exp L^p(\R^N), \; p\geq q>1.$ We aim to study local well-posedness  and look for the maximum power of the nonlinearity in terms of the existence of solutions in these spaces.  We also study the global existence for small initial data and determine the decay estimates for large time. For the global existence, we  aim to allow $f$ to behave like $|u|^{m-1}u$ near the origin, with $m>1+{2/N}.$ That is  to reach the Fujita critical exponent $1+{2/N}$.

The Orlicz space $\exp L^p(\R^N)$ is a generalization of Lebesgue spaces and contains $L^r(\R^N)$ for every $p\leq r<\infty.$ It is defined as follows
\begin{equation*}
    \exp L^p(\R^N)=\bigg\{\,u\in L^1_{loc}(\R^N);\;\int_{\R^N}\Big({\rm e}^{|u(x)|^p\over\lambda^p}-1\Big)\,dx<\infty ,\,\;\mbox{for some}\,\; \lambda>0\, \bigg\},
\end{equation*}
endowed with the Luxembourg norm
\begin{equation*}
    \|u\|_{\exp L^p(\R^N)}:=\inf\biggr\{\,\lambda>0;\,\,\,\,\int_{\R^N} \Big({\rm e}^{|u(x)|^p\over\lambda^p}-1\Big)\,dx\leq1\,\biggl\}.
\end{equation*}

Since the space of smooth compactly supported functions $C^{\infty}_0(\R^N)$ is not dense in the Orlicz space $\exp L^p(\R^N)$ (see \cite{IRT, Ioku}), we use the space $\exp L^p_0(\R^N)$  which is the closure of $C^{\infty}_0(\R^N)$ with respect to  the Luxemburg norm $\|\cdot\|_{\exp L^p(\R^N)}$.
It is known that \cite{IRT}
\begin{equation}\label{exp L0}
    \exp L^p_0(\R^N)=\biggr\{\,u\in L^1_{loc}(\R^N);\;\int_{\R^N}\Big({\rm e}^{\alpha|u(x)|^p}-1\Big)\,dx<\infty ,\,\; \mbox{for every}\,\; \alpha>0\, \biggl\}.
\end{equation}
 It is easy to show that the linear heat semi-group  ${\rm e}^{t\Delta}$ is continuous at $t=0$ in $\exp L^p_0(\R^N).$ However, this is not the case in  $\exp L^p(\R^N).$

In the sequel, we adopt the following definitions of weak, weak-mild and classical  solutions to Cauchy problem \eqref{1.1}.
\begin{definition}[{Weak solution}]
\label{weak}
Let $u_0\in\exp L^p_0(\R^N)$ and $T>0$. We say that the function $u\in C([0,T];\,\exp L^p_0(\R^N))$ is a weak solution of \eqref{1.1} if $u$ verifies \eqref{1.1} in the sense of distribution and $u(t)\to u_0$ in the weak$^*$topology as $t\searrow 0.$
\end{definition}

\begin{definition}[{Weak-mild solution}]
\label{weakmild}
We say that $u\in L^{\infty}(0,T;\, \exp L^p(\R^N))$  is a weak-mild solution of
the Cauchy problem \eqref{1.1} if $u$ satisfies the associated integral equation \eqref{integral} in $\exp L^p(\R^N)$ for almost all $t\in (0,T)$ and $u(t)\rightarrow u_0$ in the weak$^*$ topology as $t\searrow 0.$
\end{definition}

\begin{definition}[{$\exp L^p-$classical solution}]
\label{weakmild}
Let $u_0\in \exp L^p(\R^N)$ and $T>0.$ A function $u\in C((0,T];\exp L^p(\R^N))\cap L^\infty_{loc}(0,T;L^\infty(\R^N))$ is said to be $\exp L^p-$classical solution of   \eqref{1.1} if  $u\in C^{1,2}((0,T)\times \R^N)$, verifies  \eqref{1.1} in the classical sense and $u(t)\to u_0$ in the weak$^*$topology as $t\searrow 0.$
\end{definition}

We are first interested in the local well-posedness. Since $C_0^\infty(\R^N)$ is dense in $\exp L^p_0(\R^N)$, we are able to prove local existence and uniqueness to \eqref{1.1}  for initial data in $\exp L^p_0(\R^N)$. We assume that the nonlinearity $f$ satisfies

\begin{equation}\label{1.9}
    f(0)=0,\qquad |f(u)-f(v)|\leq C|u-v|({\rm e}^{\lambda \,|u|^p}+{\rm e}^{\lambda \,|v|^p}),\;\forall\;\;u,\,v\in\R,
\end{equation}
for some constants $C>0,\; p>1$ and $\lambda>0$. Our first main result reads as follows.

\begin{theorem}[{Local well-posedness}]\label{local} Suppose that $f$ satisfies \eqref{1.9}. Given any $u_0\in\exp L^p_0(\R^N)$ with $p>1,$ there exist a time $T=T(u_0)>0$ and a unique weak solution $u\in C ([0,T];\,\exp L^p_0(\R^N))$ to \eqref{1.1}.
\end{theorem}

We stress that the density of  $C_0^\infty(\R^N)$ in $\exp L^p_0(\R^N)$ is crucial in the above Theorem. In fact we have obtained the following non-existence result in $\exp L^p(\R^N)$.

\begin{theorem}[{Non-existence}]\label{nonex}
Let $p>1,\; \alpha>0$ and
\begin{equation}
\label{ID}
\Phi_\alpha(x)=\left\{\begin{array}{cc}
\alpha \Big(-\log|x|\Big)^{{1\over p}},\quad |x|<1,   \\\\
0,\quad \quad \quad \quad \quad \quad \quad |x|\geq 1.
\end{array}
\right.
\end{equation}
Assume that $f : \R \to \R$ is continuous, positive  on $[0,\infty)$ and satisfies
\begin{equation}
\label{croissancef}
 \liminf_{s\to \infty}\left(f(s)\,{\rm{e}}^{-\lambda s^p}\right)>0, \quad \lambda>0.
 \end{equation}
 Then $\Phi_\alpha\in \exp L^p(\R^N)\setminus \exp L^p_0(\R^N)$ and there exists $\alpha_0>0$ such that for every $\alpha\geq \alpha _0$ and $T>0$ the Cauchy problem \eqref{1.1} with $u_0=\Phi_\alpha $ has no nonnegative $\exp L^p-$classical solution in $[0,T].$
\end{theorem}
The results of Theorems \ref{local}-\ref{nonex} are known for $p=2$ in \cite{IRT}.

Our next interest is the global existence and the decay estimate. It depends on the behavior of the nonlinearity $f(u)$ near $u=0.$ The following behavior near $0$ will be allowed
$$|f(u)|\sim |u|^m,$$
where ${N(m-1)\over 2}\geq p.$  More precisely, we suppose that the nonlinearity $f$ satisfies
\begin{equation}\label{1.4}
f(0)=0,\qquad|f(u)-f(v)|\leq C\left|u-v\right|\bigg(|u|^{m-1}{\rm e}^{\lambda |u|^p}+|v|^{m-1}{\rm e}^{\lambda |v|^p}\bigg),\quad\forall\;u,\;v\in\R,
\end{equation}
where  ${N(m-1)\over 2}\geq p>1$, $C>0,$ and $\lambda>0$ are constants. Our aim is to obtain global existence to the Cauchy problem \eqref{1.1} for small initial data in $\exp L^p(\R^N)$. We have obtained the following.
\begin{theorem}[{Global existence}]
\label{GE} Let $N\geq 1,\; p>1,$  such that $N(p-1)/2>p.$ Assume that $m\geq p$ (hence $N(m-1)/2>p$) and the nonlinearity $f$ satisfies \eqref{1.4}. Then, there exists a positive constant $\varepsilon>0$ such that every initial data $u_0\in \exp L^p(\R^N)$
with $\|u_0\|_{\exp L^p(\R^N)} \leqslant\varepsilon,$ there exists a weak-mild solution $u\in L^{\infty}(0,\infty;\exp L^p(\R^N))$
of the Cauchy problem \eqref{1.1} satisfying
\begin{equation} \label{1.8}
    \lim_{t\longrightarrow0}\|u(t)-{\rm e}^{t\Delta}u_{0}\|_{\exp L^p(\R^N)}=0.
\end{equation}
Moreover, if $m>3/2$ then there exists a constant $C>0$ such that,
\begin{equation}
\label{Linfinibihar}
\|u(t)\|_a \leq\,C\,t^{-\sigma},\quad\forall\;t>0,
\end{equation}
where $${N(m-1)\over 2}<a<{N(m-1)\over 2}\frac{1}{(2-m)_+},\; a>N/2,\; \quad\mbox{and}\quad \sigma={1\over m-1}-{N\over 2a}>0\,.$$
\end{theorem}
\begin{remarks}\quad\\
\vspace{-0.5cm}
{\rm 
\begin{itemize}
\item[(i)] The case $N(p-1)/2\,\leq\, p$ will be investigated in a forthcoming paper.
\item[(ii)] Note that in the proof of the decay estimates, we require $a>N/2$ which is compatible with the other assumptions only if we impose the additional condition  $m>3/2.$
\item[(iii)] If only we want to prove global existence, we change the space of contraction that is we omit the Lebesgue part and we do not need such a supplementary condition on $m$.
\end{itemize}
}
\end{remarks}

Hereafter, $\|\cdot\|_r$ denotes the norm in the Lebesgue space $L^r(\R^N),\; 1\leq r\leq \infty.$ We mention that the assumption for the nonlinearity covers the cases
$$
f(u)=\pm |u|^{m-1}u{\rm e}^{|u|^p}, \quad m\geq 1+{2p\over N}.
$$
The global existence part of Theorem \ref{GE} is known for $p=2$ (see \cite{Ioku}). The estimate \eqref{Linfinibihar} was obtained in \cite{Ioku} for $p=2$ and $m=1+{4\over N}.$ This is improved in \cite{MOT} for $p=2$ and any $m\geq 1+{4\over N}.$ The fact  that estimate \eqref{Linfinibihar} depends on the smallest power of the nonlinearity $f(u)$ is known in \cite{STW} but only for nonlinearities having polynomial growth.

Using similar arguments as in \cite{Indiana}, we can show the following lower estimate of the blow-up rate.
\begin{theorem}[{Blow-up rate}]
\label{BLOW}
Assume that the nonlinearity  $f$ satisfies \eqref{1.9} with $\lambda>0$. Let $u_0\in L^p(\R^N)\cap L^\infty(\R^N)$ and $u\in C([0, T_{\max});\, L^p\cap L^\infty)$ be the maximal solution of \eqref{1.1}. If $T_{\max}<\infty$, then there exist two positive constants $C_1,\; C_2$ such that
$$
\lambda\|u(t)\|^p_{L^p\cap L^\infty}\geq  C_1\big|\log(T_{\max}-t)|+ C_2,\quad 0\leq t<T_{\max}\,.
$$
\end{theorem}

See \cite{ST} and references therein for similar blow-up estimates for parabolic problems with exponential nonlinearities.

The rest of this paper is organized as follows. In the next section, we collect some basic facts and useful tools about Orlicz spaces.  Section 3 is devoted to some crucial estimates on the linear heat semi-group. The sketches of the proofs of Theorems \ref{local} and \ref{BLOW} are done in Section 4. Section 5 is devoted to Theorem \ref{nonex} about nonexistence. Finally, in section 6 we give the proof of Theorem \ref{GE}. In all this paper, $C$ will be a positive constant which  may have different values at different places. Also, $L^r(\R^N)$, $\exp L^r(\R^N)$, $\exp L^r_0(\R^N)$ will be written respectively $L^r$, $\exp L^r$ and $\exp L^r_0$.

\section{Orlicz spaces: basic facts and useful tools}
Let us recall the definition of the so-called Orlicz spaces on $\R^N$ and some
related basic facts. For a complete presentation and more details, we refer the reader to
\cite{Adams, RR, Trud}.

\begin{definition}\label{deforl}\quad\\
Let $\phi : \R^+\to\R^+$ be a convex increasing function such that
$$
\phi(0)=0=\lim_{s\to 0^+}\,\phi(s),\quad
\lim_{s\to\infty}\,\phi(s)=\infty.
$$
We say that a  function $u\in L^1_{loc}(\R^N)$ belongs to
$L^\phi(\R^N)$ if there exists $\lambda>0$ such that
$$
\int_{\R^N}\,\phi\left(\frac{|u(x)|}{\lambda}\right)\,dx<\infty.
$$
We denote then
\begin{equation}
\label{Luxemb}
\|u\|_{L^\phi}=\inf\,\left\{\,\lambda>0,\quad\int_{\R^N}\,\phi\left(\frac{|u(x)|}{\lambda}\right)\,dx\leq
1\,\right\}.
\end{equation}
\end{definition}
It is known that $\left(L^\phi(\R^N),\|\cdot\|_{L^\phi}\right)$ is a Banach space. Note that, if $\phi(s)=s^p,\, 1\leq
p<\infty$,  then $L^\phi$ is nothing else than the  Lebesgue space $L^p$. Moreover, for $u\in L^\phi$ with $K:=\|u\|_{L^\phi}>0$,
we have
$$\left\{\,\lambda>0,\quad\int_{\R^N}\,\phi\left(\frac{|u(x)|}{\lambda}\right)\,dx\leq
1\,\right\}=[K, \infty[\,. $$
In particular
\begin{equation}
\label{med1}
\int_{\R^N}\,\phi\left(\frac{|u(x)|}{\|u\|_{L^\phi}}\right)\,dx\leq 1.
\end{equation}
We also recall the following well known properties.
\begin{proposition} \label{fatou} We have
\begin{itemize}
\item[(i)] $L^1\cap L^\infty\subset L^\phi(\R^N)\subset L^1+L^\infty$.
\item[(ii)] {\it Lower semi-continuity}:
$$
u_n\to u\quad\mbox{a.e.}\quad\Longrightarrow\quad\|u\|_{L^\phi}\leq
\liminf\|u_n\|_{L^\phi}.
$$
\item[(iii)] {\it Monotonicity}:
$$
|u|\leq|v|\quad\mbox{
a.e.}\quad\Longrightarrow\quad\|u\|_{L^\phi}\leq\|v\|_{L^\phi}.
$$
\item[(iv)] {\it Strong Fatou property}:
$$
0\leq u_n\nearrow u\quad\mbox{
a.e.}\quad\Longrightarrow\quad\|u_n\|_{L^\phi}\nearrow\|u\|_{L^\phi}.
$$
\item[(v)] {\it Strong and modular convergence}:
$$
u_n\to u \;\;\mbox{in}\;\; L^\phi \;\;\Longrightarrow\;\; \int_{\R^N}\phi(u_n-u) dx\to 0 .
$$
\end{itemize}
\end{proposition}

Denote by
$$
L_0^\phi(\R^N)=\Big\{\,u\in L^1_{loc}(\R^N),\;\;\; \int_{\R^N}\,\phi\left(\frac{|u(x)|}{\lambda}\right)\,dx<\infty,\;\forall\;\;\lambda>0\Big\}\,.
$$

It can be shown (see for example \cite{IRT}) that
$$
 L_0^\phi(\R^N)=\overline{C_0^\infty(\R^N)}^{L^\phi}= \mbox{the colsure of $C_0^\infty(\R^N)$ in $L^\phi(\R^N)$}.
 $$
Clearly $L_0^\phi(\R^N)=L^\phi(\R^N)$ for $\phi(s)=s^p, \,p\geq 1$, but this is not the case for any $\phi$ (see \cite{IRT}). When $\phi(s)={\rm e}^{s^p}-1$, we denote the space $L^\phi(\R^N)$ by $\exp L^p$ and $L^\phi_0(\R^N)$ by $\exp L^p_0$.


{ The following Lemma summarize the relationship between Orlicz and Lebesgue spaces.
\begin{lemma}\label{Orl-Leb}  We have
\begin{itemize}
\item[(i)] $\EPO\varsubsetneq \EP ,$\; $p\geq 1$.
\item[(ii)] $\EPO\not\hookrightarrow L^\infty$, hence\;$\EP\not\hookrightarrow L^\infty,$ \; $p\geq 1$.
\item[(iii)] $\EP \not\hookrightarrow L^r$,\;\; for all\;\;$1\leq r<p,$ \; $p>1$.
\item[(iv)] $L^q\cap L^\infty\hookrightarrow \EPO$,\;\; for all\;\; $1\leq q\leq p$. More precisely
\end{itemize}
\begin{equation}
\label{embed}
\|u\|_{\exp L^p}\leq \frac{1}{\left(\log 2\right)^{{1\over p}}}\biggr(\|u\|_{q}+\|u\|_{\infty}\biggl)\,.
\end{equation}
\end{lemma}

\vspace{0.3cm}
\begin{proof}[{Proof of Lemma \ref{Orl-Leb}}]
\vspace{-0.7cm}
\begin{itemize}
\item[(i)] Let $u$ be the function defined by
\begin{eqnarray*}
u(x)&=&\bigg(-\log|x|\bigg)^{1/p}\quad\mbox{if}\quad |x|\leq 1,\\
u(x)&=&0\qquad\mbox{if}\qquad |x|>1.
\end{eqnarray*}
For $\alpha>0$, we have
$$
\int_{\R^N}\,\bigg({\rm e}^{\frac{|u(x)|^p}{\alpha^p}}-1\bigg)\,dx<\infty\Longleftrightarrow \alpha>N^{-1/p}.
$$
Therefore $u\in \EP$ and $u\not\in\EPO$.
\item[(ii)]  Let $u$ be the function defined by
\begin{eqnarray*}
u(x)&=&\bigg(\log\left(1-\log|x|\right)\bigg)^{1/p}\quad\mbox{if}\quad |x|\leq 1,\\
u(x)&=&0\qquad\mbox{if}\qquad |x|>1.
\end{eqnarray*}
Clearly $u\not\in L^\infty$. Moreover, for any $\alpha>0$, we have
$$
\int_{\R^N}\,\bigg({\rm e}^{\frac{|u(x)|^p}{\alpha^p}}-1\bigg)\,dx=|\mathcal{S}^{N-1}|\int_0^1\, r^{N-1}\bigg((1-\log r)^{1\over \alpha^p} -1\bigg)\,dr <\infty,
$$
where $|\mathcal{S}^{N-1}|$ is the measure of the unit sphere $\mathcal{S}^{N-1}$ in $\R^N$. The second assertion follows since  $\EPO \hookrightarrow\EP $.
\item[(iii)] Let $u$ be the function defined by
\begin{eqnarray*}
u(x)&=&|x|^{-{N\over r}}\quad\mbox{if}\quad |x|\geq 1,\\
u(x)&=&0\qquad\mbox{if}\qquad |x|<1.
\end{eqnarray*}
Then $u\in \EPO$ but $u\not\in L^r$. Indeed, it is clear that $u\not\in L^r$, and for $\alpha>0$, we have
$$
\int_{\R^N}\,\bigg({\rm e}^{\frac{|u(x)|^p}{\alpha^p}}-1\bigg)\,dx=\frac{|\mathcal{S}^{N-1}|}{Nr}\,\sum_{k=1}^\infty\, \frac{1}{(pk-r)k!\alpha^{pk}}<\infty.
$$
\item[(iv)] Let $u\in L^q\cap L^\infty$ and let $\alpha>0$. Using the interpolation inequality
$$
\|u\|_{r}\leq \|u\|_{q}^{q/r}\,\|u\|_{\infty}^{1-{q/r}}\leq\|u\|_{q}+\|u\|_{\infty} ,\quad q\leq r\leq\infty,
$$
we obtain
\begin{eqnarray*}
\int_{\R^N}\,\bigg({\rm e}^{\frac{|u(x)|^p}{\alpha^p}}-1\bigg)\,dx&=&\sum_{k=1}^\infty\,\frac{1}{k!\alpha^{pk}}\|u\|_{L^{pk}}^{pk}\\
&\leq& \sum_{k=1}^\infty\,\frac{1}{k!\alpha^{pk}}\left(\|u\|_{q}+\|u\|_{\infty}\right)^{pk}\\
&=&  {\rm e}^{\frac{\left(\|u\|_{q}+\|u\|_{\infty}\right)^2}{\alpha^p}}-1.
\end{eqnarray*}
This clearly implies \eqref{embed}.
\end{itemize}
\end{proof}

}

 We have the embedding: $\exp L^p \hookrightarrow L^{q}$ for every $1< p\leq q$. More precisely:
\begin{lemma}\label{sarah55} For every $1\leq p\leq q<\infty,$ we have
\begin{equation}\label{2.1}
    \|u\|_{q} \leqslant\left(\Gamma\left(\frac{q}{p}+1\right)\right)^\frac{1}{q}\|u\|_{\exp L^p},
  \end{equation}
where $\Gamma(x):=\displaystyle\int_0^{\infty}\tau^{x-1}{\rm e}^{-\tau}\,d\tau, \; x>0.$
\end{lemma}
The proof of the previous lemma is similar to that in \cite{RT}. For reader's convenience, we give it here.
\begin{proof}[{Proof of Lemma \ref{sarah55}}]
Let $K=\|u\|_{\EP}>0.$ Using the inequality
$$
 \frac{|x|^{p r}}{\Gamma(r+1)}\leq {\rm e}^{|x|^p}-1,\; r\geq 1,\; x\in\R,
$$
we have
$$
\int_{\R^N} \frac{(|u(x)|/K)^{p r}}{\Gamma(r+1)}dx\leq \int_{\R^N}\left( {\rm e}^{(|u(x)|/K)^{p}}-1\right)dx\leq 1.
$$
This leads to
$$
\|u\|_{p r}\leq \left(\Gamma(r+1)\right)^{1\over p r} K.
$$
The result follows by taking $r={q\over p}\geq 1.$
\end{proof}
\begin{remark}
For $\phi(s)={\rm e}^s-1-s$, on can prove the following inequality
$$
\|u\|_q\leq C(q)\|u\|_{L^\phi},\;2\leq q<\infty,
$$
for some constant $C(q)>0$ depending only on $q$.
\end{remark}
We recall that the following properties of the functions $\Gamma$ and ${\mathcal{B}}$ given by
$$ {\mathcal{B}}(x,y)=\int_0^1\tau^{1-x}(1-\tau)^{1-y}d\tau,\quad x,\; y>0.$$ We have
\begin{equation}
\label{gamma4}
 {\mathcal{B}}(x,y)=\frac{\Gamma(x+y)}{\Gamma(x)\Gamma(y)},\; \forall\; x,\; y>0,
\end{equation}

\begin{equation}
\label{gamma1}
\Gamma(x)\geq C>0,\; \forall\; x>0,
\end{equation}

\begin{equation}
\label{gamma2}
\Gamma(x+1)\sim \left({x\over \rm{e}}\right)^x\, \sqrt{2\pi x}, \; \mbox{ as }\; x \to \infty,
\end{equation}
and
\begin{equation}
\label{gamma3}
\Gamma(x+1)\leq C x^{x+{1\over 2}},\; \forall \; x\geq 1.
\end{equation}

The following Lemmas will be useful in the proof of the global existence.
\begin{lemma}
\label{med}
Let $\lambda>0$, $1\leq p,\; q<\infty$ and $K>0$ such that $\lambda q\,K^p\leq 1$. Assume that
$$
\|u\|_{\exp L^p}\leq K\,.
$$
Then
$$
\|{\rm e}^{\lambda |u|^p}-1\|_q\leq \left(\lambda q\,K^p\right)^{{1\over q}}\,.
$$
\end{lemma}
\begin{proof}[{Proof of Lemma \ref{med}}]
Write
\begin{eqnarray*}
\int_{\R^N}\,\left({\rm e}^{\lambda |u|^p}-1\right)^q\,dx &\leq& \int_{\R^N}\,\left({\rm e}^{\lambda\,q\,|u|^p}-1\right)\,dx\\
&\leq&\int_{\R^N}\,\left({\rm e}^{\lambda\,q\,K^p\frac{|u|^p}{\|u\|_{\exp L^p}^p}}-1\right)\,dx\\
&\leq&\,\lambda\,q\,K^p\,\int_{\R^N}\,\left({\rm e}^{\frac{|u|^p}{\|u\|_{\exp L^p}^p}}-1\right)\,dx\leq \,\lambda\,q\,K^p,
\end{eqnarray*}
where we have used the fact that ${\rm e}^{\theta s}-1\leq \theta\left({\rm e}^{s}-1\right)$, $0\leq \theta\leq 1$, $s\geq 0$ and \eqref{med1}.
\end{proof}

\begin{lemma}
\label{params}
Let $m\geq p>1$, $a>\frac{N(m-1)}{2}$, $a>\frac{N}{2}$. Define
$$
\sigma=\frac{1}{m-1}-\frac{N}{2a}>0.
$$
Assume that
\begin{equation}
\label{AN}
N>\frac{2p}{p-1},
\end{equation}
and
\begin{equation}
\label{Aa}
a<\frac{N(m-1)}{2}\frac{1}{(2-m)_+}\,.
\end{equation}
Then, there exist $r,\;\;q,\;\; (\theta_k)_{k=0}^\infty\;\;, (\rho_k)_{k=0}^\infty$ such that
\begin{equation}
\label{r}
1\leq r\leq a\,.
\end{equation}

\begin{equation}
\label{q}
q\geq 1\quad \mbox{and}\quad \frac{1}{r}=\frac{1}{a}+\frac{1}{q}\,.
\end{equation}

\begin{equation}
\label{theta}
0<\theta_k<1\quad\mbox{and}\quad \frac{1}{q(pk+m-1)}=\frac{\theta_k}{a}+\frac{1-\theta_k}{\rho_k}\,.
\end{equation}

\begin{equation}
\label{rho}
 p\leq \rho_k<\infty\,.
\end{equation}

\begin{equation}
\label{beta1}
\frac{N}{2}\left(\frac{1}{r}-\frac{1}{a}\right)<1\,.
\end{equation}

\begin{equation}
\label{beta2}
 \sigma\Big[ \theta_k(pk+m-1)+1\Big]<1\,.
\end{equation}

\begin{equation}
\label{beta3}
1-\frac{N}{2}\left(\frac{1}{r}-\frac{1}{a}\right)-\sigma\theta_k(pk+m-1)=0\,.
\end{equation}
Moreover,
\begin{equation}
\label{behav1}
\theta_k\longrightarrow 0\quad\mbox{as}\quad k\longrightarrow\infty.
\end{equation}
\begin{equation}
\label{behav2}
\rho_k\longrightarrow \infty\quad\mbox{as}\quad k\longrightarrow\infty.
\end{equation}
\begin{equation}
\label{behav3}
\frac{(pk+m-1)(1-\theta_k)}{p\rho_k}\,(1+\rho_k)\leq k,\;\;\; \forall\;\;k\geq 1.
\end{equation}
\end{lemma}
\begin{remark}
The assumption \eqref{Aa} together with $a>\frac{N}{2}$ implies that $m>\frac{3}{2}$.
\end{remark}
\begin{proof}[{Proof of Lemma \ref{params}}]
Note that the assumption \eqref{Aa} implies that $\sigma<1$. It follows that, for all integer $k\geq 0$ one can choose $\theta_k$ such that
\begin{equation}
\label{choix1}
0<\theta_k<\frac{1}{pk+m-1}\min\Big(m-1, \frac{1-\sigma}{\sigma}\Big)\,.
\end{equation}
Next, we choose $\rho_k$ such that
\begin{equation}
\label{choix2}
\frac{1-\theta_k}{\rho_k}=\frac{2}{N(pk+m-1)}-\frac{2\theta_k}{N(m-1)}\,.
\end{equation}
Finally, we choose $q$ such that
\begin{equation}
\label{choix3}
\frac{1}{q(pk+m-1)}=\frac{\theta_k}{a}+\frac{1-\theta_k}{\rho_k}\,.
\end{equation}
This leads to all remainder parameters.
\end{proof}


We state the following proposition which is needed for the local well-posedness in the space $\exp L^p_0$.

\begin{proposition}\label{step0}
Let $1\leq p<\infty$ and $u\in C([0,T]; \exp L^p)$. Then for every $\alpha>0$ there holds
$$
\left({\rm e}^{\alpha |u|^p}-1\right)\in C([0,T]; L^r),\;1\leq r<\infty.
$$
\end{proposition}
\begin{proof}[Proof of Proposition \ref{step0}]
Although the proof is similar to that given in \cite{MOT}, we give it here for completeness. Using the inequality
$$
\left|{\rm e}^{ x}  -{\rm e}^{ y}\right|^r\leq \left|{\rm e}^{r x}  -{\rm e}^{ r y}\right|,\;\; x,\, y\in\R,
$$
it suffices to consider only the case $r=1$. Note that the proof for $p=2$ was done in \cite{IJMS}. The case $p=1$ follows by the inequality
$$
\left|{\rm e}^{ |x|-|y|}  -1\right|\leq  {\rm e}^{| x-y|}  -1,\;\; x,\, y\in\R,
$$
and property (v) in Proposition \ref{fatou}. The general case follows from the following lemmas.

\begin{lemma}
\label{Gen1}
Assume that
$$
v_n\to v \quad\mbox{in}\quad \exp L^p.
$$
Then, for any $\alpha>0$, we have
$$
{\rm e}^{\alpha |v_n-v|^p}-1\to 0 \quad\mbox{in}\quad L^1.
$$
\end{lemma}
\begin{proof}[{Proof of Lemma \ref{Gen1}}] It suffices to consider the case $v=0$ and $\alpha=1$. For given $0<\varepsilon\leq 1$, there exists $N\geq 1$ such that $\|v_n\|_{\EP}\leq \varepsilon$ for all $n\geq N$. By definition of the norm $\|\cdot\|_{\EP}$, there exists $0<\lambda=\lambda_n<\varepsilon$ such that
$$
\int_{\R^N}\,\left({\rm e}^{|\frac{v_n}{\lambda}|^p}-1\right)\,dx\leq 1,\quad \forall\;\;\;n\geq N.
$$
By convexity argument, we deduce that
\begin{eqnarray*}
\int_{\R^N}\,\left({\rm e}^{|v_n|^p}-1\right)\,dx&=&\int_{\R^N}\,\left({\rm e}^{\lambda^p|\frac{v_n}{\lambda}|^p}-1\right)\,dx\\
&\leq& \int_{\R^N}\,\left({\rm e}^{|\varepsilon\frac{v_n}{\lambda}|^p}-1\right)\,dx\\
&\leq&\varepsilon\int_{\R^N}\,\left({\rm e}^{|\frac{v_n}{\lambda}|^p}-1\right)\,dx\\
&\leq&\varepsilon.
\end{eqnarray*}
\end{proof}
\begin{lemma}
\label{Gen2}
Let $1<p<\infty$ and $v\in \EP$. Assume that
$$
w_n\to 0 \quad\mbox{in}\quad \EP.
$$
Then, for any $\alpha>0$, we have
$$
{\rm e}^{\alpha |w_n||v|^{p-1}}-1\to 0 \quad\mbox{in}\quad L^1.
$$
\end{lemma}
\begin{proof}[{Proof of Lemma \ref{Gen2}}]
Write
\begin{eqnarray*}
\Big\|{\rm e}^{\alpha |w_n||v|^{p-1}}-1\Big\|_{L^1}&=&\sum_{k=1}^\infty\,\frac{\alpha^k}{k!}\int\,|w_n|^k|v|^{k(p-1)}\,dx\\
&\leq&\sum_{k=1}^\infty\,\frac{\alpha^k}{k!}\|w_n\|_{L^{kp}}^k\,\|v\|_{L^{kp}}^{k(p-1)}
\end{eqnarray*}
where we have used H\"older's inequality with
$$
\frac{1}{k}=\frac{1}{kp}+\frac{1}{k\frac{p}{p-1}}.
$$
Hence, using \eqref{2.1}, we deduce that
\begin{eqnarray*}
\Big\|{\rm e}^{\alpha |w_n||v|^{p-1}}-1\Big\|_{L^1}&\leq&\sum_{k=1}^\infty\,\frac{\alpha^k}{k!}\,(k!)^{1/p}(k!)^{1-{1/p}}\,\|w_n\|_{\EP}^k\,\|v\|_{\EP}^{k(p-1)}\\
&\leq&\sum_{k=1}^\infty\,\Big(\alpha\|w_n\|_{\EP}\|v\|_{\EP}^{p-1}\Big)^k\\
&\leq&\frac{\alpha\,\|w_n\|_{\EP}\,\|v\|_{\EP}^{p-1}}{1-\alpha\,\|w_n\|_{\EP}\,\|v\|_{\EP}^{p-1}}\longrightarrow 0.
\end{eqnarray*}

\end{proof}

\begin{lemma}
\label{Gen3}
Let $1<p<\infty$. Assume that
$$
v_n\to v \quad\mbox{in}\quad \EP.
$$
Then,
$$
{\rm e}^{ |v_n|^{p}}-{\rm e}^{ |v|^{p}}\to 0 \quad\mbox{in}\quad L^1.
$$
\end{lemma}
\begin{proof}[{Proof of Lemma \ref{Gen3}}]
Set $w_n=v_n-v$, then
$$
{\rm e}^{ |v_n|^{p}}-{\rm e}^{ |v|^{p}}=\Big({\rm e}^{ |v|^{p}}-1\Big)\,\Big({\rm e}^{ |w_n+v|^{p}-|v|^p}-1\Big)+\Big({\rm e}^{ |w_n+v|^{p}-|v|^p}-1\Big).
$$
Using the following elementary inequality
$$
\exists\;\;\; \alpha>0\;\;\;\mbox{such that}\;\;\; \Big|\,|a+b|^p-|b|^p\,\Big|\,\leq\,\alpha\left(|a|^p+|a||b|^{p-1}\right),\;\;\;\forall\;\;\; a,b\in\R,
$$
it follows that
$$
\Big\|{\rm e}^{ |w_n+v|^{p}-|v|^p}-1\Big\|_{L^1}\leq \Big\|{\rm e}^{ \alpha|w_n|^{p}+\alpha|w_n||v|^{p-1}}-1\Big\|_{L^1}\,.
$$
Let us write
$$
{\rm e}^{ \alpha|w_n|^{p}+\alpha|w_n||v|^{p-1}}-1= \mathbf{I}_n+\mathbf{J}_n+\mathbf{K}_n,
$$
where
\begin{eqnarray*}
\mathbf{I}_n&=&\Big({\rm e}^{ \alpha|w_n|^{p}}-1\Big)\,\Big({\rm e}^{\alpha|w_n||v|^{p-1}}-1\Big)  \\
\mathbf{J}_n&=& \Big({\rm e}^{ \alpha|w_n|^{p}}-1\Big) \\
\mathbf{K}_n&=& \Big({\rm e}^{\alpha|w_n||v|^{p-1}}-1\Big)
\end{eqnarray*}
By Lemma \ref{Gen2} and since $w_n\to 0$ in $\EP$, $v\in \EP$, we deduce that
\begin{eqnarray*}
\mathbf{I}_n\longrightarrow 0 \quad&\mbox{in}&\quad L^1,\\
\mathbf{J}_n\longrightarrow 0 \quad&\mbox{in}&\quad L^1,\\
\mathbf{K}_n\longrightarrow 0 \quad&\mbox{in}&\quad L^1.
\end{eqnarray*}
The proof of Lemma \ref{Gen3} is complete.
\end{proof}
Combining Lemmas \ref{Gen1}-\ref{Gen2}-\ref{Gen3}, we easily deduce the desired result; that is
$$
{\rm e}^{\alpha |u|^p}-1\;\in \;C([0,T]; L^1),
$$
whenever $u\in C([0,T]; \EP)$. This finishes the proof of Proposition \ref{step0}.
\end{proof}
A straightforward consequence is:
\begin{corollary}\label{step1} Let $1\leq p<\infty$ and $u\in C([0,T]; \exp L^p)$. Assume that $f$ satisfies \eqref{1.9}. Then for every $p\leq r<\infty$ there holds
\begin{equation*}
    f(u)\in C\left([0, T]; L^r\right).
\end{equation*}
\end{corollary}
\begin{proof}
Fix $p\leq r<\infty$, $0\leq t\leq T$ and let $(t_n)\subset [0,T]$ such that $t_n\to t$. Using H\"older's inequality, we obtain
\begin{eqnarray*}
\|f(u(t_n))-f(u(t))\|_r&\leq & 2C\|u(t_n)-u(t)\|_r+\\
&&C\||u(t_n)-u(t)|({\rm e}^{\lambda \,|u(t_n)|^p}-1+{\rm e}^{\lambda \,|u(t)|^p}-1)\|_r\\
&\leq&  2C\|u(t_n)-u(t)\|_r+C\|u(t_n)-u(t)\|_{{2r}}\times \\
&&\left(\|{\rm e}^{\lambda |u(t_n)|^p}-1\|_{{2r}}+\|{\rm e}^{\lambda |u(t)|^p}-1\|_{{2r}}\right)\\
&&\hspace{-3cm}\leq C\|u(t_n)-u(t)\|_{\EP}\left(1+\|{\rm e}^{\lambda |u(t_n)|^p}-1\|_{{2r}}+\|{\rm e}^{\lambda |u(t)|^p}-1\|_{{2r}}\right),
\end{eqnarray*}
where we have used Lemma \ref{sarah55} in the last inequality. From  Proposition \ref{step0} we know that $\|{\rm e}^{\lambda |u(t_n)|^p}-1\|_{{2r}}\to \|{\rm e}^{\lambda |u(t)|^p}-1\|_{{2r}}$ as $n\to\infty$. It follows that $\|f(u(t_n))-f(u(t))\|_{r}\to 0$ which is the desired conclusion.
\end{proof}


\section{Linear estimates}

In this section we establish some results needed for the proofs of the main theorems. We first recall some basic estimates for the linear heat semigroup ${\rm e}^{t\Delta}.$   The solution of the linear heat equation
\begin{equation*}
\left\{\begin{array}{cc}
\partial_{t} u=\Delta u,\; t>0,\; x\in \R^N,\\
u(0,x)=u_{0}(x),
\end{array}
\right.
\end{equation*}
can be written as a convolution:
\begin{equation*}
    u(t,x)=\bigr(G_t\star u_0\bigl)(x):=\bigr({\rm e}^{t\Delta}u_0\bigl)(x),
\end{equation*}
where
\begin{equation*}
   G_t(x):=G(t,x)=\frac{{\rm e}^{- {|x|^2\over 4t}}}{(4\pi t)^{N\over 2}},\; t>0,\; x\in \R^N,
\end{equation*}
is the heat kernel.  We will frequently use the $L^r-L^\rho$ estimate as stated in the Proposition below.
\begin{proposition}\label{LPLQQ}  For all $1\leq r\leq \rho\leq\infty$, we have
\begin{equation}\label{x}
\|{\rm e}^{t\Delta}\varphi\|_{\rho}\leqslant   t^{-\frac{N}{2}(\frac{1}{r}-\frac{1}{\rho})}\|\varphi\|_{r},\qquad\,\, \forall\; t>0,\,\, \forall\; \varphi\in L^r.
\end{equation}
\end{proposition}

The following Proposition is a generalization of \cite[Lemma 2.2, p 1176]{Ioku}.

\begin{proposition}\label{pre} Let $1\leqslant q\leqslant p,\,\,1\leqslant r\leqslant\infty.$ Then the following estimates hold:
\begin{enumerate}
  \item [(i)] $\|{\rm e}^{t\Delta}\varphi\|_{\exp L^p}\leqslant  \|\varphi\|_{\exp L^p},\; \forall\; t>0,$\; $\forall\; \varphi\in \exp L^p.$
  \item [(ii)] $\|{\rm e}^{t\Delta}\varphi\|_{\exp L^p}\leqslant  \,t^{-\frac{N}{2q}}\left(\log(t^{-\frac{N}{2}}+1)\right)
    ^{-\frac{1}{p}} \|\varphi\|_{q},\; \forall\; t>0,$\; $\forall\; \varphi\in L^q.$
  \item [(iii)] $ \|{\rm e}^{t\Delta}\varphi\|_{\exp L^p}\leqslant  \frac{1}{\left(\log 2\right)^{{1\over p}}}\,\left[\,t^{-\frac{N}{2r}}\|\varphi\|_{r}+\|\varphi\|_{q}\right], \; \forall\; t>0,$ $\forall\; \varphi\in L^r\cap L^q.$
\end{enumerate}
\end{proposition}
\begin{proof}[{Proof of Proposition \ref{pre}}] We begin by proving (i). For any $\alpha>0,$  expanding the exponential function leads to
\begin{eqnarray*}
 \int_{\R^N}\bigg(\exp\left|\frac{{\rm e}^{t\Delta}\varphi}{\alpha}\right|^p-1\bigg)dx &=& \sum_{k=1}^{\infty}
\frac{\|{\rm e}^{t\Delta}\varphi\|^{pk}_{{pk}}}{k! \alpha^{pk}}.
\end{eqnarray*}
Then by the $L^{pk}-L^{pk}$ estimate of the heat semi-group \eqref{x}, we obtain
\begin{eqnarray*}
\int_{\R^N}\bigg(\exp\left|\frac{{\rm e}^{t\Delta}\varphi}{\alpha}\right|^p-1\bigg)dx
&\leqslant & \sum_{k=1}^{\infty}\frac{\|\varphi\|^{pk}_{{pk}}}{k! \alpha^{pk}}\\
&=&\int_{\R^N}\left(\exp\left|\frac{\varphi}{\alpha}\right|^p-1\right)dx.
\end{eqnarray*}
Therefore we obtain
\begin{eqnarray*}
  \|{\rm e}^{t\Delta}\varphi\|_{\EP} &=& \inf\left\{\alpha>0,\,\,\,\int_{\R^N}\left(\exp
  \left|\frac{{\rm e}^{t\Delta}\varphi}{\alpha}\right|^p-1\right)dx \leq1\right\} \\
  &\leqslant &\inf\left\{\alpha>0,\,\,\,\int_{\R^N}\left(\exp\left|\frac{\varphi}{\alpha}\right|^p-1\right)dx \leq1\right\}\\
  &{=}& \|\varphi\|_{\EP}.
\end{eqnarray*}
This proves (i).

We now turn to the proof of (ii). Using \eqref{x} with $q\leq p,$ we have
 \begin{eqnarray*}
 \int_{\R^N}\left(\exp\left|\frac{{\rm e}^{t\Delta}\varphi}{\alpha}\right|^p-1\right)dx
 &=&\sum_{k=1}^{\infty}\frac{\|{\rm e}^{t\Delta}\varphi\|^{pk}_{{pk}}}{k!\alpha^{pk}}\nonumber\\
 &\leqslant & \sum_{k=1}^{\infty}\frac{t^{-\frac{N}{2}(\frac{1}{q}-\frac{1}{p k})p k}\|\varphi\|^{pk}_{q}}{k!\alpha^{pk}}\nonumber\\
  &=&t^{\frac{N}{2}}\left(\exp\left(\frac{ t^{-{\frac{N}{2q}}}\|\varphi\|_{q}}{\alpha}\right)^p-1\right).
\end{eqnarray*}
It follows that
$$
  \|{\rm e}^{t\Delta}\varphi\|_{\EP} \leq t^{-\frac{N}{2q}}\left(\log(t^{-\frac{N}{2}}+1)\right)^{-\frac{1}{p}} \|\varphi\|_{q}.
$$
This proves (ii).

We now prove (iii). By the embedding $L^q\cap L^{\infty}\hookrightarrow
\EP$ \eqref{embed}, we have
\begin{eqnarray*}
   \|{\rm e}^{t\Delta}\varphi\|_{\EP}&\leqslant & {1\over (\log2)^{1/p}}\left[ \|{\rm e}^{t\Delta}\varphi\|_{{\infty}}+
   \|{\rm e}^{t\Delta}\varphi\|_{q}\right].
\end{eqnarray*}
Using the  $L^r-L^{\infty}$ estimate \eqref{x}, we get
\begin{eqnarray*}
  \|{\rm e}^{t\Delta}\varphi\|_{\EP} &\leqslant &{1\over (\log2)^{1/p}}\left[t^{-\frac{N}{2r}} \|\varphi\|_{r}+\|\varphi\|_{q}\right].
\end{eqnarray*}
This proves (iii). The proof of the proposition is now complete.
\end{proof}

As a consequence we have the following, the proof of which can be done as in \cite{MOT}.
\begin{corollary}
\label{lemme estimates}
Let $p>1,\; N> {2p\over p-1}, \; r> {N\over 2}.$ Then, for every $g\in L^1\cap L^r,$ we have
 \begin{equation*}
    \|{\rm e}^{t\Delta}g\|_{\exp L^p}\leq \kappa(t)\,\|g\|_{L^1\cap L^r},\; \forall \; t>0,
\end{equation*}
where $\kappa\in L^1(0,\infty)$ is given by
$$\kappa(t)=\frac{1}{\left(\log 2\right)^{{1\over p}}}\min\biggr\{ t^{-\frac{N}{2r}}+1,\; t^{-\frac{N}{2}}\Big(\log(t^{-\frac{N}{2}}+1)\Big)^{-\frac{1}{p}}\biggl\}.$$
\end{corollary}
Here we use $\|g\|_{L^1\cap L^q}=\|g\|_{1}+\|g\|_{q}.$
\begin{proof}[{Proof of Corollary \ref{lemme estimates}}]
We have, by Proposition \ref{pre} (ii) with $q=1$,
\begin{equation}\label{ii9}
    \|{\rm e}^{t\Delta}g\|_{\EP}\leq t^{-\frac{N}{2}}\Big(\log(t^{-\frac{N}{2}}+1)\Big)^{-\frac{1}{p}}\,\|g\|_{1}.
\end{equation}
Using Proposition \ref{pre} (iii) with $q=1$, we get
\begin{equation}\label{iii99}
    \|{\rm e}^{t\Delta}g\|_{\EP}\leq \frac{1}{\left(\log 2\right)^{{1\over p}}}\,\Big( t^{-\frac{N}{2r}}+1 \Big)\Big[\|g\|_{r}+\|g\|_{1}\Big].
\end{equation}
Combining the inequalities \eqref{ii9} and \eqref{iii99}, we obtain
\begin{equation}\label{2.8}
    \|{\rm e}^{t\Delta}g\|_{\EP}\leq \kappa(t)\;\Big(\|g\|_{1}+\|g\|_{r}\Big).
\end{equation}
By the assumption $N> {2p\over p-1},\,\,r>\frac{N}{2}$, we can see that $\kappa \in L^{1}(0,\,\infty).$
\end{proof}

We will also need the following result for the proofs.
\begin{proposition}\label{continu0}
 If $u_0\in \exp L^p_0$ then  ${\rm e}^{t\Delta}u_0 \in C([0,\infty); \exp L^p_0).$
\end{proposition}

It is known that ${\rm e}^{t\Delta}$ is a $C^0-$semigroup on $L^p.$ By Proposition \ref{continu0}, it is also a $C^0-$semigroup on $ \exp L^p_0.$ This  is not the case on $\exp L^p.$  We have the following result.

\begin{proposition}
\label{pdiscontinuite}
There exist $u_0\in \exp L^p$ and a constant $C>0$ such that
\begin{equation}
\label{ineq1}
 \|{\rm e}^{t\Delta} u_0-u_0\|_{\exp L^p} \geqslant C,\; \forall\; t>0.
\end{equation}
\end{proposition}

The proof of the previous proposition uses the notion of rearrangement of functions and can be done as in \cite{MOT}.

\section{Local well-posedness}

 In this section we prove the existence and the uniqueness of solution to \eqref{1.1} in $C([0,T]; \exp L^p_0)$ for some $T>0$, namely Theorem \ref{local}. Throughout this section we assume that the nonlinearity $f :\R\to\R$ satisfies $f(0)=0$ and
\begin{equation}\label{Nonlin}
|f(u)-f(v)|\leq C|u-v|\left({\rm e}^{\lambda \,|u|^p}+{\rm e}^{\lambda \,|v|^p}\right),\quad\forall\;u,\; v\in\R
\end{equation}
for some constants $C>0,\;\lambda>0\; p\geq 1$. We emphasize that, thanks to Proposition \ref{step1}, the Cauchy problem \eqref{1.1} admits the equivalent integral formulation \eqref{integral}. This is formulated as follows.
\begin{proposition}\label{4.1} Let $T > 0$ and $u_0$ be in $\exp L^p_0$. If $u$ belongs to $C([0,T]; \exp L^p_0)$, then $u$ is a weak solution of \eqref{1.1} if and only if $u(t)$ satisfies the integral equation \eqref{integral} for any  $t\in(0, T )$.
\end{proposition}

Now we are ready to prove Theorem \ref{local}. The idea is to split the initial data $u_0\in\EPO$ into a small part in $\EP$ and a smooth one. This will be done using the density of $C^\infty_0(\R^N)$ in $\EPO$. First we solve the initial value problem with smooth initial data to obtain a local and bounded solution $v$. Then we consider the perturbed equation satisfied by $w := u-v$ and with small initial data. Now we come to the details. For $\varepsilon>0$ to be chosen later, we write $u_0=v_0+w_0$, where $v_0\in C_0^\infty(\R^N)$ and $\|w_0\|_{\EP}\leq \varepsilon$. Then, we consider the two Cauchy problems:
\begin{equation*}
(\mathcal{P}_1)\qquad\qquad\left\{
\begin{array}{ll}
\partial_t v-\Delta v=f(v), & \qquad t>0,\,x\in\R^N,\qquad\qquad\qquad\\
v(0)=v_0,&
\end{array}
\right.
\end{equation*}

and

\begin{equation*}
(\mathcal{P}_2)\qquad\qquad\left\{
\begin{array}{ll}
\partial_t w-\Delta w=f(w+v)-f(v), & \qquad t>0,\,x\in\R^N, \\
w(0)=w_0.&
  \end{array}
\right.
\end{equation*}

We first, prove the following existence result concerning $(\mathcal{P}_1).$
\begin{proposition}\label{C0} Let $v_0\in L^p\cap L^{\infty}$. Then there exist a time $T>0$ and a solution $v\in C([0,T],\EPO)\cap L^{\infty}(0,T;L^{\infty})$ to $(\mathcal{P}_1)$.
\end{proposition}

\begin{proof}[Proof of Proposition \ref{C0}] We use a fixed point argument. We introduce, for any $M>0,$ and positive  time $T$  the following complete metric space
$${\mathcal Y}(M,T):=\Big\{\;v\in C([0,T];\EPO)\cap L^{\infty}(0,T;L^{\infty});\quad \|v\|_{T}\leq M\Big\},$$
where $\|v\|_{T}:=\|v\|_{L^{\infty}(0,T;L^{p})}+\|v\|_{L^{\infty}(0,T;L^{\infty})},$ and $\|v_0\|_{L^p\cap L^{\infty}}=\|v_0\|_{p}+\|v_0\|_{{\infty}}$.
Set
$$\Phi(v)(t):={\rm e}^{t\Delta}v_0+\int_0^{t}{\rm e}^{(t-s)\Delta}f(v(s))\,ds.$$
We will prove that, for suitable $M>0$ and $T>0,$   $\Phi$ is a contraction map from ${\mathcal Y}(M,T)$ into itself.

First, since $v_0\in L^p\cap L^{\infty},$ then by Lemma \ref{Orl-Leb} (iv), $v_0\in \EPO$ and by Proposition \ref{continu0}, ${\rm e}^{t\Delta}v_0 \in C([0,T]; \exp L^p_0).$ Obviously ${\rm e}^{t\Delta}v_0\in L^{\infty}(0,T;L^{\infty}).$ Second  by  \eqref{1.9}, $f(v)\in L^1(0,T; \EPO)$ whenever $v\in C([0,T];\EPO)\cap L^{\infty}(0,T;L^{\infty})$.
 Then, by Proposition \ref{continu0}, we conclude that $\Phi(v)\in C([0,T];\EPO)\cap L^{\infty}(0,T;L^{\infty})$.

Now, for every $v_1,\, v_2\in {\mathcal Y}(M,T)$, we have thanks to \eqref{Nonlin},
\begin{eqnarray*}
\|\Phi(v_1)-\Phi(v_2)\|_{L^{\infty}(0,T;L^q)}&\leq& C\int_0^T \|f(v_1(s))-f(v_2(s))\|_{q}\,ds\\
 &&\hspace{-2.5cm}\leq T \|f(v_1)-f(v_2)\|_{L^{\infty}(0,T;L^q)}\\
&&\hspace{-2.5cm}\leq CT\left({\rm e}^{\lambda  \|v_1\|^p_{L^\infty_t(L^\infty_x)}}+{\rm e}^{\lambda\|v_2\|^p_{L^\infty_t(L^\infty_x)}}\right)\|v_1-v_2\|_{L^\infty(0,T; L^q)}
\end{eqnarray*}
where $q=p$ or $q=\infty$. Then, it follows that
\begin{eqnarray}\nonumber
 \|\Phi(v_1)-\Phi(v_2)\|_{T}&\leq & 2C\,T\, {\rm e}^{\lambda M^p}\|v_1-v_2\|_{T}\\ \label{4.3} &\leq & 2C\,T\, {\rm e}^{\lambda M^p}\|v_1-v_2\|_{T}.
\end{eqnarray}
Similarly we have
\begin{eqnarray} \nonumber
  \|\Phi(v)\|_{T}  &\leq & \|v_0\|_{L^p\cap L^{\infty}}+C\,\,T\, {\rm e}^{\lambda M^p}\,\|v\|_{T}\\ \label{4.6} &\leq & \|v_0\|_{L^p\cap L^{\infty}}+2CM{\rm e}^{\lambda M^p}T.
\end{eqnarray}
From \eqref{4.3} and \eqref{4.6} we conclude that for $T>0$ and $M>\|v_0\|_{L^p\cap L^{\infty}}$ such that
$$2C {\rm e}^{\lambda M^p}T<1,\; \|v_0\|_{L^p\cap L^{\infty}}+2CM{\rm e}^{\lambda M^p}T\leq M,$$ $\Phi$ is a contraction map on ${\mathcal Y}(M,T)$. In particular, one can take  $M>\|v_0\|_{L^p\cap L^{\infty}}$ and $T<{M-\|v_0\|_{L^p\cap L^{\infty}}\over 2MC{\rm e}^{\lambda M^p}}.$ This finishes the proof of Proposition \ref{C0}.
\end{proof}
Following similar arguments as in \cite{MOT} and using Propositions \ref{4.1}-\ref{C0}, we end the proof of Theorem \ref{local}.\\

The solution constructed by the above Proposition  can be extended to a maximal solution by well known argument. Moreover, if $T_{\max}<\infty,$ then $\displaystyle\lim_{t\to T_{\max}}\|u(t)\|_{L^p\cap L^{\infty}} = \infty.$ Let us now give the proof of the lower blow-up estimates.

\begin{proof}[Proof of Theorem \ref{BLOW}] Let $u_0\in L^p\cap L^{\infty}$ and $u\in C([0,T_{\max}),\EPO)$  be the maximal solution of \eqref{1.1} given by Theorem \ref{local} (or Proposition \ref{C0}). To prove the lower blow-up estimates we use an argument introduced by Weissler in \cite[Section 4 and Remark (6)2]{W2}. See also \cite[Proposition 5.3, p. 901]{MW}. Assume that $T_{\max}<\infty.$ Then $\displaystyle\lim_{t\to T_{\max}}\|u(t)\|_{L^p\cap L^{\infty}} = \infty.$ Consider $u$ the solution starting at $u(t)$ for some $t\in [0,T_{\max}).$ If for some $M$ $$ \|u(t)\|_{L^p\cap L^{\infty}}+2CM{\rm e}^{\lambda M^p}(T-t)\leq M,$$
then $T<T_{\max}.$ Therefore, for any $M>0,$
 $$ \|u(t)\|_{L^p\cap L^{\infty}}+2CM{\rm e}^{\lambda M^p}(T_{\max}-t)> M.$$
Choosing $M=2\|u(t)\|_{L^p\cap L^{\infty}}$ it follows that
$$4C\|u(t)\|_{L^p\cap L^{\infty}}{\rm e}^{2^p\lambda \|u(t)\|_{L^p\cap L^{\infty}}^p}(T_{\max}-t)>\|u(t)\|_{L^p\cap L^{\infty}}.$$
That is
$${\rm e}^{2^p\lambda \|u(t)\|_{L^p\cap L^{\infty}}^p}\geq C(T_{\max}-t)^{-1},$$
for some positive constant $C.$ Hence,
$$2^p\lambda \|u(t)\|_{L^p\cap L^{\infty}}^p\geq -\log(T_{\max}-t)+C.$$
Then
$$ \lambda\|u(t)\|_{L^p\cap L^{\infty}}^p\geq -C_1\log(T_{\max}-t)+C_2$$
for some positive constants $C_1,\; C_2.$ This completes the proof of Theorem \ref{BLOW}.
\end{proof}
We obtain the following concerning problem $(\mathcal{P}_2).$
\begin{proposition}\label{UU2} Let $T>0$ and $v\in L^{\infty}(0,T;L^{\infty})$ given by Proposition \ref{C0}. Let $w_0\in\EPO$. Then for $\|w_0\|_{\EP}\leq \varepsilon$, with $\varepsilon>0$ small enough, there exist a time $\widetilde{T}=\widetilde{T}(w_0,\varepsilon,v)>0$ and a solution $w\in C([0,\widetilde{T}],\EPO)$ to problem $(\mathcal{P}_2)$.
\end{proposition}
The proof of Proposition \ref{UU2} uses the following lemma.
\begin{lemma}\label{4.3333}Let $v\in L^{\infty}$ and $w_1,\,w_2\in \EP$ with $\|w_1\|_{\EP},\, \|w_2\|_{\EP}\leq M$ for some constant $M>0$. Let $p\leq q<\infty$, and assume that $2^p\lambda q M^p\leq 1$ where $\lambda$ is given by \eqref{Nonlin}. Then there exists a constant $C>0$ such that
$$\Big\|f(w_1+v)-f(w_2+v)\Big\|_{q}\leq\,C\,{\rm e}^{2^{p-1}\lambda \|v\|_{\infty}^p}\Big\|w_1-w_2\Big\|_{\EP}.$$
\end{lemma}
\begin{proof}[Proof of the Lemma \ref{4.3333}] By the assumption \eqref{Nonlin} on $f$, we have
\begin{eqnarray*}
\Big\|f(w_1+v)-f(w_2+v)\Big\|_{q}&\leq& C \Big\||w_1-w_2|\left( {\rm e}^{2^{p-1}\lambda|w_1|^p+2^{p-1}\lambda |v|^p}+{\rm e}^{2^{p-1}\lambda|w_2|^p+2^{p-1}\lambda |v|^p}\right)\Big\|_{q}\\&&\hspace{-4cm}\leq  {\rm e}^{2^{p-1}\lambda \|v\|_{\infty}^p}\bigg(2C\Big\|w_1-w_2\Big\|_{q}+C\Big\||w_1-w_2|\left({\rm e}^{2^{p-1}\lambda |w_1|^p}-1\right)\Big\|_{q}\bigg)\\&&\hspace{-3cm}+C{\rm e}^{2^{p-1}\lambda \|v\|_{\infty}^p}\,\Big\||w_1-w_2|\left({\rm e}^{2^{p-1}\lambda |w_2|^p}-1\right)\Big\|_{q}\\&&\hspace{-4cm}\leq{\rm e}^{2^{p-1}\lambda \|v\|_{\infty}^p}\bigg(2C\Big\|w_1-w_2\Big\|_{q}+C\Big\|w_1-w_2\Big\|_{{2q}}\Big\|{\rm e}^{2^{p-1}\lambda |w_1|^p}-1\Big\|_{{2q}}\bigg)\\&&\hspace{-3cm}+C{\rm e}^{2^{p-1}\lambda \|v\|_{\infty}^p}\,\Big\|w_1-w_2\Big\|_{{2q}}\Big\|{\rm e}^{2^{p-1}\lambda |w_2|^p}-1\Big\|_{{2q}}\\&&\hspace{-4cm}\leq\,C\,{\rm e}^{2^{p-1}\lambda \|v\|_{\infty}^p}\Big\|w_1-w_2\Big\|_{\EP},
\end{eqnarray*}
where we have used H\"older inequality, Lemma \ref{sarah55}, Lemma \ref{med}  and the fact that
$(a+b)^p\leq 2^{p-1}(a^p+b^p)$, for every $a, b\geq 0$ and any $p\geq 1.$ This finishes the proof of Lemma \ref{4.3333}.
\end{proof}


\section{Non-existence}

The  following lemma is the key of the proof of Theorem \ref{nonex}.
\begin{lemma}
\label{NON}
Let $p>1,\,\alpha>0.$ Let $\Phi_\alpha$ be given by \eqref{ID} and $f$, $\lambda>0$ be as in \eqref{croissancef}. Then, there exists $\alpha_0>0$ such that for any $\alpha\geq \alpha_0$, $\varepsilon>0$ and $r>0$, we have
$$
\int_0^\varepsilon\,\int_{|x|<r}\,\exp\left(\lambda\,\biggl({\rm e}^{t\Delta}\Phi_\alpha\biggr)^p\right)\,dx\,dt=\infty\,.
$$
\end{lemma}
\begin{proof}[{Proof of Lemma \ref{NON}}]
 Let $B(a, \rho)$ denotes the open ball centered at $a\in\R^N$ and with radius $\rho>0.$ Fix $\varepsilon,\, r>0$. For $\rho=\min\left(r, {1\over 4}\right)$, we have $B(3x, |x|)\subset B(0, 1)$ for any $|x|<\rho$. Therefore, for any $|x|<\rho$, it holds
\begin{eqnarray*}
\biggl({\rm e}^{t\Delta}\Phi_\alpha\biggr)(x)&=&\frac{1}{(4\pi t)^{N/2}}\,\int_{|x|<1}\,{\rm e}^{-{|x-y|^2\over 4t}}\,\Phi_\alpha(y)\,dy\\
&\geq&
\frac{\alpha}{(4\pi t)^{N/2}}\,\int_{|y-3x|<|x|}\, {\rm e}^{-{|x-y|^2\over 4t}}\,\Big(-\log|y|\Big)^{{1\over p}}\,dy\\&\geq&
C\alpha\left({|x|^2\over t}\right)^{N/2}\,{\rm e}^{-{9\over 4}{|x|^2\over t}}\,\Big(-\log 4|x|\Big)^{1/p}\,.
\end{eqnarray*}
Let $\eta=\min\left(\varepsilon, \rho^2\right)$. Then, for any $0<t<\eta$, we have $B(0, \sqrt{t})\subset B(0,\rho)$. Hence
\begin{eqnarray*}
\int_0^\varepsilon\,\int_{|x|<r}\,\exp\left(\lambda\,\biggl({\rm e}^{t\Delta}\Phi_\alpha\biggr)^p\right)\,dx\,dt&\geq& \int_0^\eta\,\int_{|x|<\rho}\,\exp\left(\lambda\,\biggl({\rm e}^{t\Delta}\Phi_\alpha\biggr)^p\right)\,dx\,dt\\&\geq&
\int_0^\eta\,\int_{{\sqrt{t}\over 2}<|x|<\sqrt{t}}\,\exp\left(-C\lambda\alpha^p\log(4|x|)\right)\,dx\,dt\\&\geq&
C_\alpha\,\int_0^\eta\, t^{{N\over 2}-{C\lambda\alpha^p\over 2}}\,dt=\infty,
\end{eqnarray*}
for $\alpha\geq \alpha_0:=\left({N+2\over C\lambda}\right)^{1/p}.$ This finishes the proof of Lemma \ref{NON}.
\end{proof}

The proof of Theorem \ref{nonex} follows similar arguments as in \cite{IRT} and uses the previous Lemma.
\section{Global Existence}


This section is devoted to the proof of Theorem \ref{GE}. The proof uses a fixed point argument on the associated integral equation
 \begin{equation}\label{NN9}
   u(t)= {\rm e}^{t\Delta}u_{0}+\int_{0}^{t}{\rm e}^{(t-s)\Delta}(f(u))(s) ds,
\end{equation}
 where  $\|u_0\|_{\EP}\leq \varepsilon$, with  small $\varepsilon>0$ to be fixed later. The  nonlinearity $f$ satisfies $f(0)=0$ and
\begin{equation}\label{NNN9}
|f(u)-f(v)|\leq C \left|u-v\right|\bigg(|u|^{m-1}{\rm e}^{\lambda |u|^p}+|v|^{m-1}{\rm e}^{\lambda |v|^p}\bigg),
\end{equation}
for some constants $C>0$ and $\lambda>0,\, p\geq 1$ and $m$ is larger than $1+{2p\over N}.$  From \eqref{NNN9}, we obviously deduce that
\begin{equation}
\label{taylorm}
 |f(u)-f(v)|\leq C|u-v|\sum_{k=0}^{\infty} \frac{\lambda^{k}}{k!} \bigg(|u|^{pk+m-1}+|v|^{pk+m-1}\bigg).
\end{equation}
We will perform a fixed point argument on a suitable metric space.  For $M>0$ we  introduce the space
$$Y_M :=\left\{u\in L^\infty(0,\infty, \EP);\;\displaystyle\sup_{t>0}  t^{\sigma}\|u(t)\|_{a}+\|u\|_{L^{\infty}(0,\infty;\EP)}\leq M\right\},$$ where $a>{N(m-1)\over 2}\geq p$ and $$\sigma={1\over m-1}-\frac{N}{2a}={N\over 2}\left({2\over N(m-1)}-{1\over a}\right)>0.$$ Endowed with the metric $d(u,v)=\displaystyle\sup_{t>0} \Big(t^{\sigma}\|u(t)-v(t)\|_{r}\Big)$, $Y_M$ is a complete metric space. This follows by Proposition \ref{fatou}.

For $u\in Y_M,$ we define $\Phi(u)$ by
\begin{equation}\label{Phi}
    \Phi(u)(t):={\rm e}^{t\Delta}u_{0}+\int_{0}^{t}{\rm e}^{(t-s)\Delta}(f(u(s))) ds.
\end{equation}

By Proposition \ref{pre} (i), Proposition \ref{LPLQQ} and Lemma \ref{sarah55}, we have
$$\|{\rm e}^{t\Delta}u_{0}\|_{\EP}\leq \|u_{0}\|_{\EP},$$
and
\begin{eqnarray*}
t^\sigma \|{\rm e}^{t\Delta}u_{0}\|_a&\leq & t^\sigma t^{-{N\over 2}\left({2\over N(m-1)}-{1\over a}\right)}\|u_{0}\|_{{N(m-1)\over 2}}
\\ &= & \|u_{0}\|_{{N(m-1)\over 2}}\leq C\|u_{0}\|_{\EP},
\end{eqnarray*}
where we have used $1\leq p\leq {N(m-1)\over 2}< a.$

Let $u\in Y_M$. Using Proposition \ref{pre} and Corollary \ref{lemme estimates}, we get for $q>{N\over 2}$,
\begin{eqnarray*}
\|\Phi(u)(t)\|_{\EP}&\leq&\|{\rm e}^{t\Delta}u_{0}\|_{\EP}+\int_{0}^{t}\left\|{\rm e}^{(t-s)\Delta}
(f(u(s))) \right\|_{\EP}\,ds\\
&\leq&\|{\rm e}^{t\Delta}u_{0}\|_{\EP}+\int_{0}^{t}\kappa(t-s)\bigg(\|f(u(s))\|_{L^1\cap L^q}
\bigg)\,ds\\
&\leq&\|{\rm e}^{t\Delta}u_{0}\|_{\EP}+ \|f(u)\|_{L^{\infty}(0,\infty;(L^1\cap L^q))}\int_{0}^{\infty}\kappa(s)\,ds\\
&\leq& \|{\rm e}^{t\Delta}u_{0}\|_{\EP}+C \|f(u)\|_{L^{\infty}(0,\infty;(L^1\cap L^q))}.
\end{eqnarray*}
Hence by Part (i) of Proposition \ref{pre}, we get
$$
\|\Phi(u)\|_{L^\infty(0{,\infty; \EP)}}\leq \|u_{0}\|_{\EP}+ C \|f(u)\|_{L^{\infty}(0,\infty;\,L^1\cap L^q)}.
$$
It remains to estimate the nonlinearity $f(u)$ in $L^r$ for $r=1,\,q.$ To this end, let us remark that
\begin{equation}\label{3.4}
|f(u)|
\leq C|u|^m\left({\rm e}^{\lambda |u|^p}-1\right)+C|u|^m.
\end{equation}
By H\"{o}lder's inequality and Lemma \ref{sarah55}, we have for $1\leq r\leq q$ and since $m\geq p$,
\begin{eqnarray}\label{3.6}
\nonumber
\|f(u)\|_{{r}} &\leq & C\|u\|_{{mr}}^m +C\||u|^m({\rm e}^{\lambda |u|^p}-1)\|_{{r}}\\ & \leq & C\|u\|_{{mr}}^m +C\|u\|_{{2mr}}^m\|{\rm e}^{\lambda |u|^p}-1\|_{{2r}} \\ &\leq& \nonumber C\|u\|_{\EP}^m\bigg(\|{\rm e}^{\lambda |u|^p}-1\|_{{2r}}+1\bigg).
\end{eqnarray}
According to Lemma \ref{med},  and the fact that $u\in Y_M$, we have for $2q\lambda M^p\leq 1$,
\begin{equation}
\label{3.8}
\|f(u)\|_{L^\infty(0,\infty; L^{r})}\leq CM^m.
\end{equation}

 Finally, we obtain
\begin{eqnarray*}
\|\Phi(u)\|_{L^\infty(0,\infty, \EP)}&\leq&  \|u_{0}\|_{\EP}+C M^m\\
&\leq& \varepsilon+CM^m.
\end{eqnarray*}

Let  $u,\,v$ be two elements of $Y_M.$ By using \eqref{taylorm} and Proposition \ref{LPLQQ}, we obtain
\begin{eqnarray*}
 t^{\sigma}\|\Phi(u)(t)-\Phi(v)(t)\|_{a}&\leq&
t^{\sigma}\int_{0}^{t}\left\| {\rm e}^{(t-s)\Delta}
(f(u(s))-f(v(s)))\right\|_{a} ds\\&\leq& t^{\sigma}
\int_{0}^{t}(t-s)^{-{N\over 2}(\frac{1}{r}-\frac{1}{a})}\left\|f(u(s))-f(v(s))\right\|_{r}\,ds\\
&&\hspace{-3cm}\leq C\sum_{k=0}^{\infty}\frac{\lambda^{k}}{k!} t^{\sigma}\int_{0}^{t}(t-s)^{-{N\over 2}(\frac{1}{r}-\frac{1}{a})}
\|(u-v)(|u|^{pk+m-1}+|v|^{pk+m-1})\|_{r}ds,
\end{eqnarray*}
where $1\leq r\leq a.$ We use the H\"{o}lder inequality with ${1\over r}={1\over a}+{1\over q}$ to obtain
\begin{eqnarray*}
 t^{\sigma}\|\Phi(u)(t)-\Phi(v)(t)\|_{a}&\leq& C\sum_{k=0}^{\infty}\frac{\lambda^{k}}{k!} t^{\sigma}\int_{0}^{t}(t-s)^{-{N\over 2}(\frac{1}{r}-\frac{1}{a})}
\|u-v\|_{a}\times\\
&&\||u|^{pk+m-1}+|v|^{pk+m-1}\|_{q}ds,\\
&\leq&  C\sum_{k=0}^{\infty}\frac{\lambda^{k}}{k!} t^{\sigma}\int_{0}^{t}(t-s)^{-{N\over 2}(\frac{1}{r}-\frac{1}{a})}
\|u-v\|_{a}\times\\ &&\left(\|u\|^{pk+m-1}_{{q(pk+m-1)}}+\|v\|^{pk+m-1}_{{q(pk+m-1)}}\right)ds.
\end{eqnarray*}

Using interpolation inequality where $\frac{1}{q(pk+m-1)}=\frac{\theta}{a}+\frac{1-\theta}{\rho},\; p\leq \rho<\infty,$  we find that
\begin{equation*}
\begin{split}
t^{\sigma}\left\|\int_0^t{\rm e}^{(t-s)\Delta}\,(f(u)-f(v))\,ds\right\|_{a}\,\,\!\!\!\!
&\leq C\sum_{k=0}^{\infty}\frac{\lambda^{k}}{k!}t^{\sigma}\int_0^t\,(t-s)^{-\frac{N}{2}
(\frac{1}{r}-\frac{1}{a})}\|u-v\|_{a}\\
&\hspace{-5.5cm}\times\biggl(\|u\|^{(pk+m-1)\theta}_{{a}}
  \|u\|^{(pk+m-1)(1-\theta)}_{{\rho}}
  +\|v\|^{(pk+m-1)\theta}_{{a}}\|v\|^{(pk+m-1)(1-\theta)}_{{\rho}}\biggr)\,ds.
\end{split}
\end{equation*}
By Lemma \ref{sarah55}, we obtain
\begin{eqnarray}\label{4.16b}
&&t^{\sigma}\left\|\int_0^t{\rm e}^{(t-s)\Delta}\,(f(u)-f(v))\,ds\right\|_{a}\nonumber\\
&&\qquad\qquad\leq C\sum_{k=0}^{\infty}\frac{\lambda^{k}}{k!}t^{\sigma}
\int_0^t\,(t-s)^{-\frac{N}{2}(\frac{1}{r}-\frac{1}{a})}\|u-v\|_{a}\Gamma\left(\frac{\rho}{p}+1\right)
^{\frac{(pk+m-1)(1-\theta)}{\rho}}\qquad\qquad\nonumber\\
&&\qquad\times\left(\|u\|^{(pk+m-1)\theta}_{{a}}\|u\|^{(pk+m-1)(1-\theta)}_{\EP}+
\|v\|^{(pk+m-1)\theta}_{{a}}\|v\|^{(pk+m-1)(1-\theta)}_{\EP}\right)\,ds.
\end{eqnarray}
Applying the fact that $u,\; v\in\; Y_M $ in  \eqref{4.16b}, we see that
\begin{eqnarray}\label{4.18sb}
&&t^{\sigma}\left\|\int_0^t{\rm e}^{(t-s)\Delta}\,(f(u)-f(v))\,ds\right\|_{a}\nonumber\\
&&\qquad\leq Cd(u,v)\sum_{k=0}^{\infty}\frac{\lambda^{k}}{k!}\Gamma\left(\frac{\rho}{p}+1\right)^{\frac{(pk+m-1)(1-\theta)}{\rho}} M^{pk+m-1}
\nonumber\\
&&\qquad \qquad\qquad\times t^{\sigma}\bigg(\int_0^t(t-s)^{-\frac{N}{2}(\frac{1}{r}-\frac{1}{a})}s^{-\sigma(1+(pk+m-1)\theta)}\,ds\bigg)\nonumber\\
&&\qquad\leq Cd(u,v)\,\sum_{k=0}^{\infty}\frac{\lambda^{k}}{k!}\Gamma\left(\frac{\rho}{p}+1\right)^{\frac{(pk+m-1)(1-\theta)}{\rho}} M^{pk+m-1}
\qquad\qquad\qquad\qquad\qquad\nonumber\\
&&\qquad\qquad\qquad\times {\mathcal{B}}\left(1-\frac{N}{2}\left(\frac{1}{r}-\frac{1}{a}\right),1-\sigma\big(1+(pk+m-1)\theta\big)\right),
\end{eqnarray}
where the parameters $a,\,q,\,r,\,\theta=\theta_k,\,\rho=\rho_k$  are given by Lemma \ref{params}.
For these parameters, using \eqref{gamma4} and \eqref{gamma1}, we obtain that
\begin{equation}\label{4.20b}
{\mathcal{B}}\left(1-\frac{N}{2}\left(\frac{1}{r}-\frac{1}{a}\right),1-\sigma\big(1+(pk+m-1)\theta\big)\right)\leq C .
\end{equation}
Moreover, using \eqref{behav1}-\eqref{behav2}-\eqref{behav3} together with \eqref{gamma3} and \eqref{gamma2} gives
\begin{equation}\label{4.21b}
   \Gamma\left(\frac{\rho_k}{p}+1\right)^{\frac{(pk+m-1)(1-\theta_k)}{\rho_k}} \leq C^k k !.
\end{equation}
 Combining \eqref{4.18sb}, \eqref{4.20b} and \eqref{4.21b} we get
$$
t^{\sigma}\left\|\int_0^t{\rm e}^{(t-s)\Delta}\,(f(u)-f(v))\,ds\right\|_{a}\leq C d(u,v) \sum_{k=0}^{\infty} \,{(C\lambda)^k} M^{pk+m-1}.$$
Hence, we get for $M$ small,
\begin{equation*}\label{4.18sss}
t^{\sigma}\left\|\int_0^t{\rm e}^{(t-s)\Delta}\,(f(u)-f(v))\,ds\right\|_{a}\leq C M^{m-1} d(u,v).
\end{equation*}

The above estimates show that $\Phi : Y_M \to Y_M $ is a contraction mapping.  By Banach's fixed point theorem, we thus obtain the existence of a unique
 $u$ in  $Y_M$ with $\Phi(u)=u.$ By \eqref{Phi}, $u$ solves the integral equation \eqref{NN9} with $f$ satisfying \eqref{NNN9}. The estimate \eqref{Linfinibihar} follows from $u\in Y_M.$ This terminates the proof of the existence of a  global solution to \eqref{NN9} for $N>\frac{2p}{p-1}$.


We will now prove the statement \eqref{1.8}. For $q\geq\frac{N}{2}$ and $q\geq p$, we have
\begin{eqnarray}\label{V.549}
 && \|u(t)-{\rm e}^{t\Delta}u_0\|_{\EP} \qquad\qquad\qquad\qquad\nonumber\\
 &&\qquad\leq \int_0^t\|{\rm e}^{(t-s)\Delta}
  f(u(s))\|_{\EP}\,ds\nonumber\\
  &&\qquad\leq C\int_0^t\|{\rm e}^{(t-s)\Delta}f(u(s))\|_{p}ds+C\int_0^t\|{\rm e}^{(t-s)\Delta}f(u(s))\|_{{\infty}}\,ds\nonumber\\
  &&\qquad\leq C\int_0^t\|f(u(s))\|_{p}ds+C\int_0^t(t-s)^{-\frac{N}{2q}} \|f(u(s))\|_{q}\,ds.
\end{eqnarray}
Now, let us  estimate $\|f(u(t))\|_{r}$ for $r=p,\,q.$  We have
\begin{equation*}
    |f(u)|\leq C|u|^m{\rm e}^{\lambda |u|^p}.
\end{equation*}
Therefore, we obtain
\begin{equation*}
 \|f(u)\|_{r}\leq C \||u|^m({\rm e}^{\lambda |u|^p}-1+1)\|_{r}.
\end{equation*}
By using H\"{o}lder inequality and Lemma \ref{sarah55}, we obtain
\begin{eqnarray*}
  \|f(u)\|_{r}&\leq& C\|u\|^{m}_{{2mr}}\|{\rm e}^{\lambda |u|^p}-1\|_{{2r}}+\|u\|^{m}_{{mr}}\\
&\leq&C\|u\|^{m}_{\EP}\left(\|{\rm e}^{\lambda |u|^p}-1\|_{{2r}}+1\right).
\end{eqnarray*}
Using Lemma \ref{med} we conclude that
\begin{equation}\label{V.559}
\|f(u)\|_{r}\leq C \|u\|^{m}_{\EP} \left((2\lambda r M^p)^{\frac{1}{2r}}+1\right)\leq C \|u\|^{m}_{\EP}.
\end{equation}
Substituting \eqref{V.559} in \eqref{V.549}, we have
\begin{eqnarray*}
\|u(t)-{\rm e}^{t\Delta}u_0\|_{\EP}&\leq& C\int_0^t \bigg[\|u\|^{m}_{\EP}+(t-s)^{-\frac{N}{2q}}
\|u\|^{m}_{\EP}\bigg]ds\\
&\leq& Ct \|u\|^{m}_{L^{\infty}(0,\infty;\,\EP)}+Ct^{1-\frac{N}{2q}}\|u\|^{m}_{L^{\infty}(0,\infty;\,\EP)} \\
&\leq &C_1t+C_2t^{1-\frac{N}{2q}},
\end{eqnarray*}
where $C_1,\,C_2$ are  finite positive constants. This gives
\begin{equation*}
    \lim_{t\longrightarrow0}\|u(t)-{\rm e}^{t\Delta}u_0\|_{\EP}=0,
\end{equation*}
and proves statement \eqref{1.8}.

Finally the fact that $u(t)\to u_0$ as $t\to 0$ in the weak$^*$ topology can be done as in \cite{Ioku}. So we omit the proof here.



\end{document}